\documentclass[11pt,reqno]{amsart}
\usepackage[margin=3cm]{geometry}
\usepackage{amsmath,amssymb,amscd}
\usepackage{amsthm}
\usepackage{mathtools} 
\usepackage{enumitem}
\usepackage[normalem]{ulem} 
\usepackage{xcolor, graphicx}
\usepackage{tikz}
\usepackage{ifthen}
\usetikzlibrary{patterns,decorations.pathreplacing}

\newcommand{\sa}{s\alpha}
\newcommand{\N}{\mathbb{N}}
\newcommand{\Xhat}{\widehat{X}}

\newcommand{\conc}{{}^\smallfrown}

\newcommand{\sN}{\mathcal{N}}
\newcommand{\R}{\mathbb{R}}
\newcommand{\len}{\mathrm{len}}

\theoremstyle{plain}
\newtheorem{theorem}{Theorem}
\newtheorem{lemma}[theorem]{Lemma}
\newtheorem{proposition}[theorem]{Proposition}

\theoremstyle{definition}

\newtheorem{example}[theorem]{Example}
\newtheorem{remark}[theorem]{Remark}

\title{A non-Borel special alpha-limit set in the square}
\author{Stephen Jackson}
\address{Department of Mathematics, University of North Texas, General Academics Building 435, 1155 Union Circle, \#311430, Denton, Texas 76203-5017}
\email{stephen.jackson@unt.edu}
\author{Bill Mance}
\address{Uniwersytet im. Adama Mickiewicza w Poznaniu, Collegium Mathematicum, ul. Umultowska 87, 61-614 Pozna\'{n}, Poland}
\email{william.mance@amu.edu.pl}
\author{Samuel Roth}
\address{Mathematical Institute of the Silesian University in Opava, Na Rybni\v{c}ku 1, 74601, Opava, Czech Republic}
\email{samuel.roth@math.slu.cz}

\makeatletter
\@namedef{subjclassname@2020}{%
  \textup{2020} Mathematics Subject Classification}
\makeatother
\subjclass[2020]{Primary: 03E15; 37E99 Secondary: 37B10}
\keywords{special alpha limit set, triangular map of the square, non-Borel analytic set}
\thanks{The first author was supported by NSF grant DMS-1800323. The second author was supported by grant 2019/34/E/ST1/00082 for the project ``Set theoretic methods in dynamics and number theory,'' NCN (The National Science Centre of Poland). The third author was supported by RVO (Czech Republic) funding for I\v{C}47813059.}

\begin{document}
\begin{abstract}
We consider the complexity of special $\alpha$-limit sets, a kind of backward limit set for non-invertible dynamical systems. We show that these sets are always analytic, but not necessarily Borel, even in the case of a surjective map on the unit square. This answers a question posed by Kolyada, Misiurewicz, and Snoha.
\end{abstract}

\maketitle



A \emph{topological dynamical system} is a pair $(X,f)$ where $X$ is a
compact metric space and $f\in C(X,X)$ is a continuous map of $X$ to
itself. Limit sets and backward limit sets provide some of the most important tools in
understanding the behavior of a topological dynamical system, since
they provide information about the long-term behavior of the orbits of
the system. One notion, in particular, of a backward limit set is the notion of
a {\em special $\alpha$-limit set}, which has played an important role in one-dimensional
dynamics (the notion is defined and discussed below).

Descriptive set theory provides a means for calibrating
the complexity of sets in Polish spaces, and this in turn is important for understanding the impact of
possible theorems about the objects under consideration. To give an example
of this phenomenon, consider the case of normal numbers for some base $b \geq 2$.
Let $\sN_b$ denote the set of real numbers which have a normal base-$b$ expansion.
Let $\sN^\perp_b$ denote the set of real numbers $\gamma$ such that
$\gamma+N_b \subseteq \sN_b$, that is, they preserve base-$b$ normality under addition.
By its definition, this set does not at first even appear to be a Borel set as it
involves a universal quatification over $\sN_b$ (the definition shows it to be
a coanalytic, or ${\boldsymbol{\Pi}}^1_1$ set; see \cite{Kechris} for the definition). However, a deep theorem of 
Rauzy characterizes the
normality-preserving numbers as those with zero upper
noise~\cite{Rauzy}. Since the upper noise is defined by taking limits
(and lim sup's) of sequences of continuous functions of $\gamma$, it
follows immediately that $\sN_b^\perp$ is Borel, in fact a ${\boldsymbol{\Pi}}^0_3$ set.
In fact, in \cite{AJM} it was shown that $\sN^\perp_b$ is a ${\boldsymbol{\Pi}}^0_3$-complete
set, which shows that it is no simpler than this. 

The present paper shows that the special $\alpha$-limit set of a point in
a topological dynamical system, in fact even for the case of the unit square in
$\R^2$, can be a ${\boldsymbol{\Sigma}}^1_1$-complete set, and thus not Borel.
The significance of this is that it tells us that there no 
such ``hidden theorems'' for special $\alpha$-limit sets, even  in
compact metric spaces. Thus, the definition involving an existential quantification over 
all backward orbits cannot be simplified.

The authors have recently shown that special $\alpha$-limit sets can in fact occur at any level in the Borel hierarchy, but this will be part of a forthcoming paper.

We now discuss these notions more precisely. 
The \emph{$\omega$-limit set} of a point $x\in X$ under a map $f$, denoted
$\omega(x)$ or $\omega(x,f)$, consists of all accumulation points of the \emph{forward
  orbit} $(f^n(x))_{n\in\N}$ of the point $x$. These sets are always closed, and their topological properties are well understood, see eg.~\cite{ABCP, AC, BlockCoppel, BrucknerSmital, JimSmi,Sharkovsky}.


The backward limit sets of a point $x$ play a dual role to the $\omega$-limit set. The idea is that to describe the ``source'' of a trajectory, we should reverse the direction of time. For invertible dynamical systems (homeomorphisms) this leads to a well-defined $\alpha$-limit set $\alpha(x)=\omega(x,f^{-1})$. This idea came into discrete dynamics from the study of flows, where $\alpha$-limit sets are fundamental in defining such objects as unstable manifolds, homoclinic and heteroclinic trajectories, and the Morse decompositions at the heart of Conley theory~\cite{Conley, GuckHol}.

For non-invertible mappings $f:X\to X$, the situation is more complicated, since a point may have several preimages (or none at all). There are several ways to define backward limit sets in this setting:

\begin{enumerate}[wide, labelwidth=!, labelindent=0pt]

\item The \emph{$\alpha$-limit set of a point}, denoted $\alpha(x)$, consists of all accumulation points of all \emph{preimage sequences} $(x_n)_{n\in\N}$, where $x_n\in f^{-n}(\{x\})$ for all $n$. These sets are especially useful in dimension one. Coven and Nitecki proved that a point is non-wandering for an interval map if and only if it belongs to its own $\alpha$-limit set~\cite{CovNit}, and Cui and Ding related the $\alpha$-limit sets of a unimodal interval map to its renormalizations~\cite{CuiDing}.

\item The \emph{$\alpha$-limit set of a backward orbit}, denoted $\alpha((x_n)_{n=0}^\infty)$, consists of all accumulation points of a single \emph{backward orbit}, i.e. a sequence $(x_n)_{n\in\N}$, where $f(x_{n+1})=x_n$ for all $n$. Hirsch et al. showed that the $\alpha$-limit set of a backward orbit is always internally chain transitive~\cite{Hirsch}. A converse result holds when $f$ is expansive and has the shadowing property~\cite{Good}. For maps on the interval, the $\alpha$-limit set of a backward orbit satisfies a local transportation condition which makes it simultaneously an $\omega$-limit set of $f$~\cite{Bal}. 

\item The \emph{special $\alpha$-limit set of a point}, denoted $\sa(x)$, is the union $\bigcup \alpha((x_n)_{n=0}^\infty)$ taken over all \emph{backward orbits of $x$}, i.e.\ sequences $(x_n)_{n\in\mathbb{N}}$ such that $f(x_{n+1})=x_n$ for all $n$ and $x_0=x$. These sets were defined by Hero, who showed that for interval maps, a point is in the attracting center if and only if it belongs to its own special $\alpha$-limit set~\cite{Hero}. Generalizations of this result to other one-dimensional spaces were given in~\cite{SXCZ, SXL}.
\end{enumerate}

Kolyada et al.\ pointed out that special $\alpha$-limit sets need not be closed, and asked whether they are necessarily Borel or even analytic~\cite{KMS}. The difficulty arises when $x$ has uncountably many backward orbit branches, since we are then taking an uncountable union of their (closed) accumulation sets.

The complexity of special $\alpha$-limit sets (or $\sa$-limit sets, for short) for maps of the interval $I=[0,1]$ will be addressed in a forthcoming paper, where it will be shown that $\sa(x)$ is always Borel, and in fact both $F_\sigma$ and $G_\delta$~\cite{HantakRoth}.

This paper is concerned with the complexity of $\sa$-limit sets in other compact metric spaces. We start in Section~\ref{sec:analytic} by showing that $\sa$-limit sets are always analytic. The proof is short and uses the fact that backward orbits occur in a well-known compact metric space related to $(X,f)$, namely, the natural extension.

The main construction in this paper gives a map on the unit square $I^2=[0,1]^2$ with a $\sa$-limit set which is ${\boldsymbol{\Sigma}}^1_1$-complete, i.e.\ analytic but not Borel. Our construction starts in Section~\ref{sec:shift} at the symbolic level with a one-sided shift space $X\subset\{0,1,2,3,4\}^\N$, in which the set of all ill-founded trees on a countable set have been suitably encoded into the backward orbit branches of a given point $x$. We show that $\sa(x)\subset X$ is not Borel. Then in Section~\ref{sec:square} we embed this shift space as a totally invariant subsystem for a map on the square $F:I^2\to I^2$, and show that the $\sa$-limit set of the corresponding point is not a Borel subset of $I^2$. We remark that the map $F$ in our counterexample is surjective, piecewise monotone, and triangular, i.e.~a skew product map of the form $F(x,y)=(f(x),g_x(y))$.

\section{Special $\alpha$-limit sets are analytic}\label{sec:analytic}

A \emph{Polish space} is any separable, completely metrizable topological space. A subset $B$ of a Polish space $Y$ is called \emph{analytic} (or a ${\boldsymbol{\Sigma}}^1_1$ set) if there exist a Polish space $X$, a Borel subset $A\subset X$, and a continuous map $f:X\to Y$ such that $B=f(A)$. The class of analytic sets in an uncountable Polish space strictly contains the class of Borel sets. A set is called \emph{${\boldsymbol{\Sigma}^1_1}$-complete} if it is analytic but not Borel, or equivalently, if it is analytic but its complement is not.

Recall that a topological dynamical system is a pair $(X,f)$ where $X$ is a compact metric space and $f:X\to X$ is continuous. In this section we show that $\sa$-limit sets of topological dynamical systems are always analytic.

\begin{theorem}\label{th:analytic}
Every $\sa$-limit set of a topological dynamical system $(X,f)$ is analytic.
\end{theorem}
\begin{proof}
Give $X^\N$ the product topology. It is compact and metrizable. For example, if $d$ is the metric on $X$, then $\widehat{d}(x_0x_1\cdots,y_0y_1\cdots) = \sum_{n} 2^{-n} \min(1,d(x_n,y_n))$ is one compatible metric on $X^\N$. Now let $\Xhat\subset X^\N$ be the closed subspace
\begin{equation*}
\Xhat = \{ x_0 x_1 \cdots ~|~ f(x_{n+1})=x_n \text{ for all } n\in\N\}.
\end{equation*}
$\Xhat$ is a topological space whose points are the backward orbits of $(X,f)$. We remark that the space $\Xhat$ equipped with the map $x_0x_1\cdots \mapsto f(x_0)x_0x_1\cdots$ is a well-known object in topological dynamics called the natural extension of $(X,f)$.

Consider the relation $R\subset \Xhat\times X$ given by
\begin{align*}
R &= \{ (x_0x_1\cdots, y) ~|~ x_{n_i} \to y \text{ along some subsequence } n_i\to\infty \} \\
&= \bigcap_{j\in\N} \bigcap_{N\in\N} \bigcup_{n=N}^\infty \left\{ (x_0x_1\cdots, y) ~|~ d(x_n,y) < \tfrac{1}{j} \right\}.
\end{align*}
Then $R$ is a countable intersection of open sets, so it is Borel.

Let $\pi_0:\Xhat \to X$ be the projection on the zeroth coordinate, $x_0x_1\cdots \mapsto x_0$. Now fix an arbitrary point $x\in X$ and let $A=\pi_0^{-1}(\{x\})$. Then $A$ is closed in $\Xhat$. Finally, let $\pi:\Xhat\times X \to X$ be the map $(x_0x_1\cdots,y)\mapsto y$. By the definition of $\sa$-limit sets we have
\begin{equation*}
\sa(x)=\pi(R\cap(A\times X)).
\end{equation*}
Thus we have expressed $\sa(x)$ as the continuous image of a Borel subset of the compact metric space $\Xhat\times X$. It follows that $\sa(x)$ is analytic.
\end{proof}

\section{A non-Borel $\sa$-limit set in a shift space}\label{sec:shift}

A \emph{word} of length $n$ (in the alphabet $\N$) is a finite sequence $s=(s_0,\ldots,s_{n-1})\in \N^n$ and an infinite word is a sequence $s=(s_0,s_1,\ldots)\in\N^\N$. We write $\len(s)\in\N\cup\{\infty\}$ for the length of a word $s$. Given a natural number $i\leq\len(s)$, we write $s|i=(s_0,s_1,\ldots,s_{i-1})$ and call this an \emph{initial segment} of $s$. We write $t\subseteq s$ to indicate that $t$ is an initial segment of $s$. The set of all finite-length words is denoted $\N^{<\N}$. This includes the \emph{empty word} $\emptyset$ of length 0, which is an initial segment of every word.

If $s, t$ are words and $\len(s)<\infty$, then we write $s\conc t$ for the \emph{concatenation} of these words, where $(s\conc t)_i=s_i$ for $0\leq i<\len(s)$ and $(s\conc t)_i=t_{i-\len(s)}$ for $\len(s)\leq i< \len(s)+\len(t)$. Given $a\in \N$ we write $a^n$ for the word $(a,\ldots,a)$ of length $n$. Similarly, we write $a^\infty$ for the infinite word $(a,a,\ldots)$. When no confusion can result, we will sometimes drop the parentheses and commas from the notation (especially when we use only single-digit natural numbers) and write a word simply as a block of symbols, eg.\ $s=(4,4,0)=440$ defines a word of length $3$.

Let $\mathcal{P}(\N^{<\N})$ denote the power set of $\N^{<\N}$, i.e.\ the set of all subsets of $\N^{<\N}$. This set becomes a Polish space when we give it the topology for which the \emph{cylinder sets} $N_s=\{T\in\mathcal{P}(\N^{<\N})~:~s\in T\}$, $s\in\N^{<\N}$ form a basis. Note that the cylinder sets are both open and closed in this topology.

A \emph{tree} (on the set of natural numbers) is any set of words $T\in\mathcal{P}(\N^{<\N})$ which is closed with respect to initial segments, i.e.\ if $t\in T$ and $s\subseteq t$, then $s\in T$ as well. An \emph{infinite branch} of a tree $T$ is any infinite word $y$ such that $y|n \in T$ for all $n\in\N$. The \emph{body} of $T$ is the set $[T] \subset \N^\N$ of all its infinite branches.

Let $IF$ be the set of all trees on $\N$ which have an infinite branch, the so-called \emph{ill-founded} trees. Let $R\subset \mathcal{P}(\N^{<\N}) \times \N^\N$ be the set of all $(T,y)$ for which $T$ is an ill-founded tree and $y$ is an infinite branch of $T$. It is easy to see that $R$ is closed and $IF=\pi(R)$, where $\pi:\mathcal{P}(\N^{<\N}) \times \N^\N \to \mathcal{P}(\N^{<\N})$ is the projection on the first coordinate. This shows that $IF$ is an analytic set in the Polish space $\mathcal{P}(\N^{<\N})$. But it is known that $IF\subset \mathcal{P}(\N^{<\N})$ is not Borel; in fact, this is regarded as the prototypical example of a ${\boldsymbol{\Sigma}}^1_1$-complete set, see~\cite{Kechris}.

We wish to ``encode'' the set $IF$ into a $\sa$-limit set. Let $2^\N$ be the standard Cantor space. Choose an enumeration $(s_i)$ of $N^{<\N}$ and define a homeomorphism $h:\mathcal{P}(\N^{<\N}) \to 2^{\N}$ by letting $h(T)$ be the point $x$ such that $x_i=1$ whenever $s_i\in T$ and $x_i=0$ whenever $s_i\not\in T$. Let $5^\N=\{0,1,2,3,4\}^\N$ be the one-sided shift space in five symbols (it is also a Cantor space) and let $\sigma:5^\N\to 5^\N$ be the shift map defined by $(\sigma x)_i=x_{i+1}$. A \emph{subshift} in $5^\N$ is any closed subset $X\subseteq 5^\N$ such that $\sigma(X)\subseteq X$. For $T$ an ill-founded tree and $y$ an infinite branch of $T$ we define inductively points $\omega_n(T,y)\in5^\N$ by
\begin{equation}\label{omega}
\begin{aligned}
\omega_0(T,y)&=3\conc 4\conc 0^\infty,\\
\omega_{n}(T,y)&=3\conc (x|n) \conc 2^{n+y_{n-1}} \conc \omega_{n-1}(T,y), \qquad \text{for $n\geq 1$, where $x=h(T)$}.
\end{aligned}
\end{equation}
Then define a subset $X\subset5^\N$ by
\begin{equation}\label{X}
X:=\{0,1,2,3\}^\N \cup \{\sigma^j(\omega_i(T,y))~|~(T,y)\in R, i,j\in\N\}.
\end{equation}

\begin{example}
Let $T$ be the tree of strictly increasing finite words $s\in\N^{<\N}$ and let $y=(2,3,5,7,11,\ldots)\in \N^\N$ be an infinite word in the body of $T$. Let $(s_i)$ be an enumeration of $\N^{<\N}$ whose first few terms are $s_0=\emptyset$, $s_1=(0)$, $s_2=(0,0)$, so that $s_0, s_1\in T$ but $s_2\not\in T$. Then
\begin{equation*}
\omega_3(T,y)=3\;110\;22222222\; 3\;11\;22222\; 3\;1\;222\; 3\; 4\; 00000\cdots
\end{equation*}
Thus, from $\omega_n(T,y)$ we can read off the first $n$ symbols of $x=h(T)$ and $y$. Each block of 2's encodes a digit of $y$. The blocks of 0's and 1's record the initial segments of $x$.
\end{example}

\begin{lemma}\label{lem:X}
The set $X\subset5^\N$ is a subshift and $\sigma(X)=X$.
\end{lemma}
\begin{proof}
Invariance $\sigma(X)\subseteq X$ is immediate from~\eqref{X}. Surjectivity of $\sigma|_X$ follows from~\eqref{omega},~\eqref{X}, and the observation that $\sigma^{t_n}(\omega_n(T,y))=\omega_{n-1}(T,y)$, where $t_n=2n+1+y_{n-1}$.

Equations~\eqref{omega} and~\eqref{X} implicitly determine a set of forbidden words $F$ in the alphabet $\{0,1,2,3,4\}$. Any word with an initial 4 and a non-zero symbol occuring after it is forbidden. Any word which ends with a 4 is forbidden if it is not an initial segment of some $\sigma^j(\omega_i(T,y))$, $i,j\in\N$, $(T,y)\in R$. A point $x\in5^\N$ is in $X$ if and only if it contains none of those forbidden words. This shows that $X$ is closed.
\end{proof}

Now consider the dynamical system given by the subshift $\sigma|_X : X \to X$ and fix the point $\omega_0=3\conc4\conc0^\infty\in X$. The backward orbit branches of $\omega_0$ correspond to the ill-founded trees, allowing us to prove the following result.

\begin{theorem}\label{th:shift-nonBorel}
The set $\sa(\omega_0)\subset X$ is ${\boldsymbol{\Sigma}}^1_1$-complete.
\end{theorem}
\begin{proof}
We know from Theorem~\ref{th:analytic} that $\sa(\omega_0)$ is analytic. It remains to show that it is not Borel. Define a map $f:\mathcal{P}(\N^{<\N}) \to X$ by $f(T)=3\conc h(T)$. The map is well-defined since $X$ contains $\{0,1,2,3\}^\N$. It is continuous since $h$ is. We will show that $T\in\mathcal{P}(\N^{<\N})$ is an ill-founded tree if and only if $f(T)\in\sa(\omega_0)$. This means $IF=f^{-1}(\sa(\omega_0))$, and since the preimage of a Borel set by a continuous map is always Borel, it follows that $\sa(\omega_0)$ is not a Borel subset of $X$.

Suppose first that $T\in \mathcal{P}(\N^{<\N})$ is an ill-founded tree. Let $y\in\N^\N$ be an infinite branch of $T$. Then $(T,y)\in R$ and so each of the points $\omega_n(T,y)$, $n\in\N$, belongs to $X$. Clearly there is a backward orbit branch of $\omega_0$ with $(\omega_n(T,y))_n$ as a subsequence. But $\omega_n(T,y) \to 3\conc h(T) = f(T)$ as $n\to \infty$.

Conversely, choose any $T\in \mathcal{P}(\N^{<\N})$ such that $f(T)\in\sa(\omega_0)$. Then we can choose a backward orbit branch of $\omega_0$ with a subsequence converging to $f(T)$, so we get points $\omega_i\in X$ and times $t_i\geq 1$, $i\in\N$, such that $\omega_i\to 3\conc h(T)$ as $i\to\infty$ and
\begin{equation}\label{shift}
\sigma^{t_i}(\omega_{i+1})=\omega_i, \qquad i\in\N.
\end{equation}
We may suppose without loss of generality that each $\omega_i$ begins with the symbol 3. By the definition of $X$ we may write $\omega_i = \omega_{n_i}(T_i,y_i)$ for some $(T_i,y_i)\in R$, $n_i\in\N$. We warn the reader that $y_i\in\N^\N$ is a whole infinite word; the subscript here does not refer to the symbol in position $i$, but to the $i$th infinite word in the sequence. Since $n_i$ counts the number of 3's in $\omega_i$ (except the last 3 before the 4), equations~\eqref{omega} and~\eqref{shift} imply that $n_{i+1} > n_i$. Thus the numbers $n_i$ form an increasing sequence, and in particular $n_i\geq i$ for all $i$. Write $x_i=h(T_i)$ for all $i$. Again from~\eqref{omega} and~\eqref{shift} we get $x_i|n_i=x_{i+1}|n_i$ and $y_i|n_i=y_{i+1}|n_i$ for all $i$. Therefore we may pass to the limits $x=\lim x_i$, $y=\lim y_i$, and since $\omega_i\to 3\conc h(T)$ we get $x=h(T)$. But $h$ is a homeomorphism and therefore $T=\lim T_i$. Since each $(T_i,y_i)\in R$ and $R$ is closed, we conclude $(T,y)\in R$, that is, $T$ is an ill-founded tree. This concludes the proof.
\end{proof}

\section{A non-Borel $\sa$-limit set in the square}\label{sec:square}

The \emph{square} is the Cartesian product $I^2=[0,1]\times[0,1]$ with the usual Euclidean topology. For a mapping $f:X\to X$ a subset $A\subset X$ is called \emph{invariant} if $f(A)\subset A$ and \emph{totally invariant} if $f^{-1}(A)=A$. Two topological dynamical systems $(X,f)$, $(Y,g)$ are called \emph{conjugate} if there is a homeomorphism $h:X\to Y$ such that $h\circ f=g\circ h$. The main goal of this section is to prove the following theorem.

\begin{theorem}\label{th:sq-nonBorel}
There is a continuous surjective map $F : I^2 \to I^2$ and a point $x_0\in I^2$ such that $\sa(x_0)\subset I^2$ is ${\boldsymbol{\Sigma}}^1_1$-complete.
\end{theorem}

The proof procedes by constructing the map $F$. It has the form of a skew product $F(x,y)=(f(x),g_x(y))$. There is a closed subset $S$ in the ``$x$-axis'' of the square, $S\subset I\times\{0\}$, such that $F^{-1}(S)=F(S)=S$. The subsystem $(S,F)$ is conjugate to the subshift $(X,\sigma)$ constructed in the previous section.

\begin{proposition}\label{prop:embedding}
There is a surjective continuous map $F:I^2\to I^2$ and a closed subset $S\subset I^2$ such that $F^{-1}(S)=F(S)=S$ and $(S,F)$ is conjugate to $(X,\sigma)$.
\end{proposition}

\begin{proof}
Start with the embedding $e:5^\N \to I$ and the map $f:I\to I$ given by
\begin{gather*}
e(x_0x_1x_2\cdots)=\sum_{j=0}^\infty \frac{2x_j}{9^{j+1}},\\
f(x)=\begin{cases}
9x - 2i, & \text{ for } x\in\left[\frac{2i}{9},\frac{2i+1}{9} \right], i\in\{0,1,2,3,4\} \\[.5em]
(2i+2) - 9x & \text{ for } x\in\left[\frac{2i+1}{9},\frac{2i+2}{9} \right], i\in\{0,1,2,3\}.
\end{cases}
\end{gather*}
Then $f$ is the standard $9$-horseshoe with $5$ increasing laps and $4$ decreasing laps and $C=e(5^\N)\subset I$ is the Cantor set of points whose trajectories stay in the increasing laps, see Figure~\ref{fig:skew1}. Clearly $f\circ e = e\circ \sigma$, and so $(C,f)$ is conjugate to $(5^\N,\sigma)$. Henceforth we will write $I_i=[\frac{2i}{9},\frac{2i+1}{9}]$ for the $i$th increasing lap of $f$.

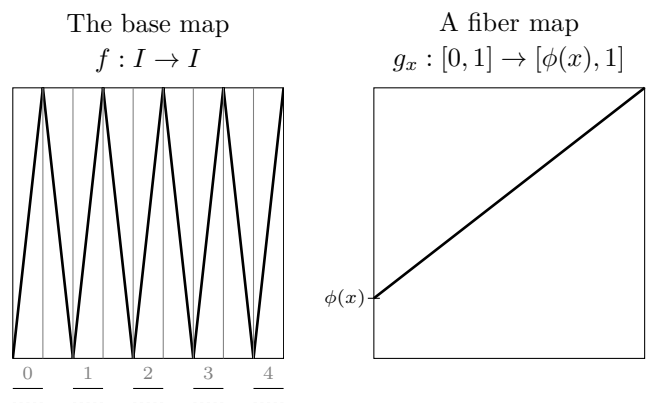
\begin{figure}[htb!!]
\begin{tikzpicture}[scale=0.4]
\begin{scope}
\clip (0,0) rectangle (9,9);
\draw (0,0) -- (9,0) -- (9,9) -- (0,9) -- cycle;
\draw[line width=1] (0,0) -- (1,9) -- (2,0) -- (3,9) -- (4,0) -- (5,9) -- (6,0) -- (7,9) -- (8,0) -- (9,9);
\foreach \i in {1,...,8} {\draw[gray] (\i,0) -- (\i,9);}
\end{scope}
\foreach \i in {0,...,4}{\node[gray] at (2*\i+0.5,-0.5) {\tiny$\i$};}
\node[align=center] at (4.5,10.5) {\small The base map\\\small$f:I\to I$};
\foreach \i in {0,...,4} {\draw (2*\i,-1) -- (2*\i+1,-1);}
\foreach \i in {0,...,4} \foreach \j in {0,...,4} {\draw (2*\i+2*\j/9,-1.5) -- (2*\i+2*\j/9+1/9,-1.5);}
\begin{scope}[shift={(12,0)}]
\draw (0,0) -- (9,0) -- (9,9) -- (0,9) -- cycle;
\draw[line width=1] (0,2) -- (9,9);
\draw (-0.2,2) node[inner sep=0pt] [label={[inner sep=0pt]left:{\tiny$\phi(x)$}}] {}-- (0.2,2);
\node[align=center] at (4.5,10.5) {\small A fiber map\\\small$g_x:[0,1] \to [\phi(x),1]$};
\end{scope}
\end{tikzpicture}
\caption{The components of the skew product mapping $F(x,y)=(f(x),g_x(y))$. The invariant Cantor set $C$ is suggested by the line segments below the graph of $f$. For the definitions of $\phi$ and $g_x$ see equations~\eqref{phi},~\eqref{F}.}
\label{fig:skew1}
\end{figure}

By Lemma~\ref{lem:X} we have $\sigma(X)=X$. That means we can write $X$ as the following finite union of closed sets
\begin{equation*}
X=\bigcup_{i<5} X_i, \quad X_i:=\{x\in X~|~i\conc x \in X\}.
\end{equation*}
Here $X_i$ is the ``follower set'' of the symbol $i$ in $X$, i.e. the set of points in $X$ which can be preceded by the symbol $i$ and still belong to $X$.

Going through the embedding $e$ we get that the closed set $A=e(X)\subset I$ is the union of the closed sets $A_i=e(X_i)$, $i<5$. Moreover, we have $A\subset C$ and $(A,f)$ is conjugate to $(X,\sigma)$. For any closed set $Y\subset I$ we write $d(x,Y)=\min\{|x-y|~:~y\in Y\}$ for the distance from a point $x\in I$ to $Y$. For $i<5$ we define sets
\begin{equation*}
B_i=\{x\in I ~:~ d(x,A_i) \leq d(x,A_j) \text{ for all }j<5\}.
\end{equation*}
Thus $B_i$ represents the points $x\in I$ for which $A_i$ is as close or closer to $x$ than any of the other sets $A_j$. Clearly the sets $B_i$ are closed and their union is all of $I$. For points $x\in A$ we have $x\in B_i$ if and only if $x\in A_i$, because the distance to at least one of the sets $A_i$ is zero. This shows that
\begin{equation}\label{BcapA}
I=\bigcup_{i<5} B_i, \text{ and for all $i<5$, }B_i\cap A = A_i.
\end{equation}

Now we define a function $\phi:I\to [0,\tfrac12]$. Let
\begin{equation}\label{phi}
\phi(x)=\frac12 d(f(x),B_i), \quad \text{for $x\in I_i$, $i<5$}.
\end{equation}
This defines $\phi$ on the increasing laps of $f$. On the decreasing laps of $f$ we have already defined $\phi$ at the endpoints, and we simply extend to any continuous function with $0<\phi(x)\leq \frac12$ whenever $x\in I\setminus\bigcup_{i<5} I_i$.

Now define a skew product mapping $F:I^2\to I^2$ by
\begin{equation}\label{F}
F(x,y)=\big(f(x),g_x(y)\big), \quad \text{where } g_x(y)=\phi(x)+y\cdot(1-\phi(x)),
\end{equation}
see Figures~\ref{fig:skew1} and~\ref{fig:skew2}.

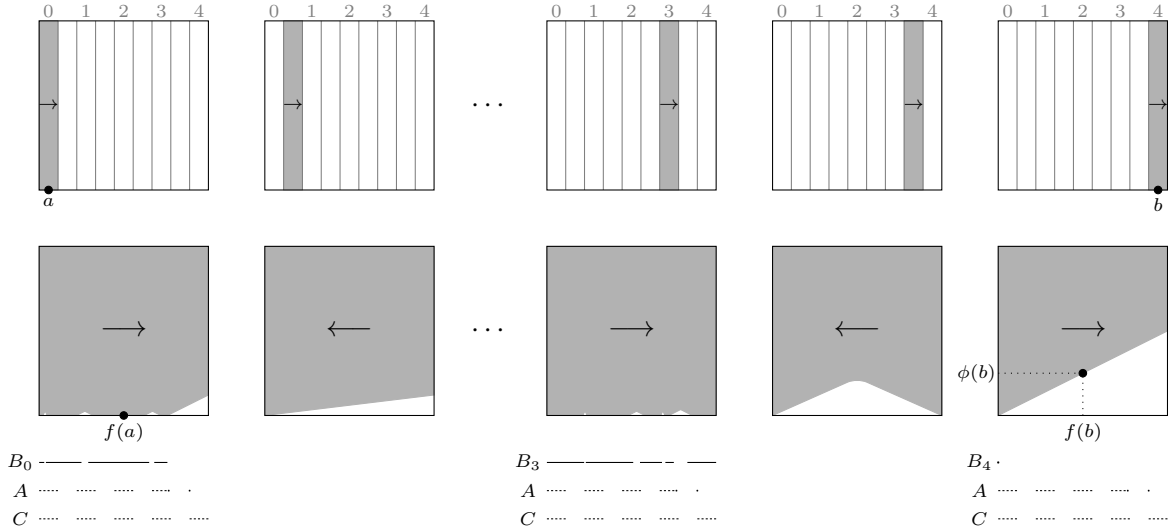
\begin{figure}[htb!!]
\begin{tikzpicture}[scale=0.25]
\draw[fill,gray!60] (0,0) rectangle (1,9);
\draw[fill] (0.5,0) circle [radius=0.2] node [inner sep=0] [label={[inner sep=2pt]below:{\tiny$a$}}] {};
\node at (0.5,4.5) {\tiny$\rightarrow$};
\foreach \i in {1,...,8}{\draw[gray] (\i,0) -- (\i,9);}
\foreach \i in {0,...,4}{\node[gray] at (2*\i+0.5,9.5) {\tiny$\i$};}
\draw (0,0) -- (9,0) -- (9,9) -- (0,9) -- cycle;
\begin{scope}[shift={(0,-12)}]
\draw[fill,gray!60] (0,0) -- 
(9*0.22,0) -- (9*0.27,9*0.025) -- (9*0.32,0) --
(9*0.62,0) -- (9*0.67,9*0.025) -- (9*0.72,0) --
(9*491/648,0) -- (9,9*1/2*157/648) -- (9,9) -- (0,9) -- cycle;
\draw[fill,white] (9*0.03,0) -- (9*0.035,0.1) -- (9*0.04,0) -- cycle;
\node at (4.5,4.5) {$\longrightarrow$};
\draw[fill] (4.5,0) circle [radius=0.2] node [inner sep=0] [label={[inner sep=2pt]below:{\tiny$f(a)$}}] {};
\draw (0,0) -- (9,0) -- (9,9) -- (0,9) -- cycle;
\end{scope}
\begin{scope}[shift={(0,-17.5)}]
\foreach \i in {0,...,4} \foreach \j in {0,...,4} {\draw (2*\i+2*\j/9,0) -- (2*\i+2*\j/9+1/9,0);}
\foreach \i in {0,...,3} \foreach \j in {0,...,4} {\ifthenelse{\i<3 \OR \j<4}{\draw (2*\i+2*\j/9,1.5) -- (2*\i+2*\j/9+1/9,1.5);}{}}
\draw (8,1.5) circle [radius=0.02];
\draw (6+8/9,1.5) circle [radius=0.02];
\draw (0,3) -- 
(9*0.03,3) (9*0.04,3) --
(9*0.25,3) (9*0.29,3) -- 
(9*0.65,3) (9*0.68,3) -- 
(9*491/648,3);
\node at (-1,0) {\tiny$C$};
\node at (-1,1.5) {\tiny$A$};
\node at (-1,3) {\tiny$B_0$};
\end{scope}

\begin{scope}[shift={(12,0)}]
\draw[fill,gray!60] (1,0) rectangle (2,9);
\node at (1.5,4.5) {\tiny$\rightarrow$};
\foreach \i in {1,...,8}{\draw[gray] (\i,0) -- (\i,9);}
\foreach \i in {0,...,4}{\node[gray] at (2*\i+0.5,9.5) {\tiny$\i$};}
\draw (0,0) -- (9,0) -- (9,9) -- (0,9) -- cycle;
\begin{scope}[shift={(0,-12)}]
\draw[fill,gray!60] (0,0) -- (9,9*1/2*157/648) -- (9,9) -- (0,9) -- cycle;
\node at (4.5,4.5) {$\longleftarrow$};
\draw (0,0) -- (9,0) -- (9,9) -- (0,9) -- cycle;
\end{scope}
\end{scope}

\node at (24,4.5) {$\cdots$};
\node at (24,-7.5) {$\cdots$};

\begin{scope}[shift={(27,0)}]
\draw[fill,gray!60] (6,0) rectangle (7,9);
\node at (6.5,4.5) {\tiny$\rightarrow$};
\foreach \i in {1,...,8}{\draw[gray] (\i,0) -- (\i,9);}
\foreach \i in {0,...,4}{\node[gray] at (2*\i+0.5,9.5) {\tiny$\i$};}
\draw (0,0) -- (9,0) -- (9,9) -- (0,9) -- cycle;
\begin{scope}[shift={(0,-12)}]
\draw[fill,gray!60] (0,0) -- 
(9*0.48,0) -- (9*0.53,9*0.025) -- (9*0.58,0) --
(9*0.72,0) -- (9*0.79,9*0.035) -- (9*0.86,0) --
(9,0) -- (9,9) -- (0,9) -- cycle;
\draw[fill,white] (9*0.22,0) -- (9*0.225,0.1) -- (9*0.23,0) -- cycle;
\draw[fill,white] (9*0.68,0) -- (9*0.685,0.1) -- (9*0.69,0) -- cycle;
\node at (4.5,4.5) {$\longrightarrow$};
\draw (0,0) -- (9,0) -- (9,9) -- (0,9) -- cycle;
\end{scope}
\begin{scope}[shift={(0,-17.5)}]
\foreach \i in {0,...,4} \foreach \j in {0,...,4} {\draw (2*\i+2*\j/9,0) -- (2*\i+2*\j/9+1/9,0);}
\foreach \i in {0,...,3} \foreach \j in {0,...,4} {\ifthenelse{\i<3 \OR \j<4}{\draw (2*\i+2*\j/9,1.5) -- (2*\i+2*\j/9+1/9,1.5);}{}}
\draw (8,1.5) circle [radius=0.02];
\draw (6+8/9,1.5) circle [radius=0.02];
\draw (0,3) --
(9*0.22,3) (9*0.23,3) --
(9*0.51,3) (9*0.55,3) --
(9*0.68,3) (9*0.7,3) --
(9*0.75,3) (9*0.83,3) -- (9,3);
\node at (-1,0) {\tiny$C$};
\node at (-1,1.5) {\tiny$A$};
\node at (-1,3) {\tiny$B_3$};
\end{scope}
\end{scope}

\begin{scope}[shift={(39,0)}]
\draw[fill,gray!60] (7,0) rectangle (8,9);
\node at (7.5,4.5) {\tiny$\rightarrow$};
\foreach \i in {1,...,8}{\draw[gray] (\i,0) -- (\i,9);}
\foreach \i in {0,...,4}{\node[gray] at (2*\i+0.5,9.5) {\tiny$\i$};}
\draw (0,0) -- (9,0) -- (9,9) -- (0,9) -- cycle;
\begin{scope}[shift={(0,-12)}]
\draw[fill,gray!60] [rounded corners] (0,0) -- (4.5,2) [sharp corners] -- (9,0) -- (9,0) -- (9,9) -- (0,9) -- cycle;
\node at (4.5,4.5) {$\longleftarrow$};
\draw (0,0) -- (9,0) -- (9,9) -- (0,9) -- cycle;
\end{scope}
\end{scope}

\begin{scope}[shift={(51,0)}]
\draw[fill,gray!60] (8,0) rectangle (9,9);
\node at (8.5,4.5) {\tiny$\rightarrow$};
\foreach \i in {1,...,8}{\draw[gray] (\i,0) -- (\i,9);}
\foreach \i in {0,...,4}{\node[gray] at (2*\i+0.5,9.5) {\tiny$\i$};}
\draw (0,0) -- (9,0) -- (9,9) -- (0,9) -- cycle;
\draw[fill] (9*17/18,0) circle [radius=0.2] node [inner sep=0] [label={[inner sep=2pt]below:{\tiny$b$}}] {};
\begin{scope}[shift={(0,-12)}]
\draw[fill,gray!60] (0,0) -- (9,4.5) -- (9,9) -- (0,9) -- cycle;
\node at (4.5,4.5) {$\longrightarrow$};
\draw (0,0) -- (9,0) -- (9,9) -- (0,9) -- cycle;
\draw[fill] (1/2*9,1/4*9) circle [radius=0.2];
\draw[dotted] (4.5,0) -- (4.5,2.25);
\node at (4.5,0) [inner sep=0] [label={[inner sep=2pt]below:{\tiny$f(b)$}}] {};
\draw[dotted] (0,2.25) -- (4.5,2.25);
\node at (0,2.25) [inner sep=0] [label={[inner sep=0]left:{\tiny$\phi(b)$}}] {};
\end{scope}
\begin{scope}[shift={(0,-17.5)}]
\foreach \i in {0,...,4} \foreach \j in {0,...,4} {\draw (2*\i+2*\j/9,0) -- (2*\i+2*\j/9+1/9,0);}
\foreach \i in {0,...,3} \foreach \j in {0,...,4} {\ifthenelse{\i<3 \OR \j<4}{\draw (2*\i+2*\j/9,1.5) -- (2*\i+2*\j/9+1/9,1.5);}{}}
\draw (8,1.5) circle [radius=0.02];
\draw (6+8/9,1.5) circle [radius=0.02];
\draw (0,3) circle [radius=0.03];
\node at (-1,0) {\tiny$C$};
\node at (-1,1.5) {\tiny$A$};
\node at (-1,3) {\tiny$B_4$};
\end{scope}
\end{scope}
\end{tikzpicture}\\
\caption{The map $F:I^2\to I^2$ applied to nine rectangular regions. Arrows indicate orientation. $C$ is the full Cantor set $5^\N$ embedded into $I$. $A$ is the embedded copy of the shift space $X$. The sets $B_i$ have gaps around the points that cannot be preceded by the symbol $i$. For example, $B_0$ is missing the embedded image of $40^\infty$, since this point can only be preceded by a $3$, and $B_4$ is a singleton, since only $0^\infty$ can be preceded by a $4$. (Note: the gaps are not drawn to scale, and many smaller gaps are not visible at all). The choice of $\phi$ causes the image of each rectangle $I_i\times I$ to meet the $x$-axis of the square only in the set $B_i$. The point $a$ represents $02^\infty$, so $a,f(a)\in A$ and $\phi(a)=0$. The point $b$ represents $42^\infty$, so $b\not\in A$, $f(b)\in A$, and $\phi(b)>0$.} 
\label{fig:skew2}
\end{figure}

Let $E:5^\N\to I^2$ be the embedding $E(x)=(e(x),0)$ and let $S=E(X)=A\times\{0\}$. We claim that $\phi$ has the following three properties:

\begin{enumerate}[label=(\roman*)]
\item $\phi(x)=0$ for all $x\in A$,
\item $\phi(x)>0$ if $x\not\in A$ but $f(x)\in A$,
\item For each $y\in I$ there is a preimage $x\in f^{-1}(y)$ with $\phi(x)=0$.
\end{enumerate}

Property (i) gives $F\circ E=E \circ\sigma$, which implies that $F(S)=S$ and $(S,F)$ is conjugate to $(X,\sigma)$. Property (ii) implies that $F^{-1}(S)=S$. Property (iii) makes $F$ surjective. This is everything we needed to show about the system $(S,F)$. It remains to establish properties (i) -- (iii) of the function $\phi$.

To prove property (i) choose $x\in A$ and write $x=e(x_0x_1\cdots)$ with $x_0x_1\ldots\in X$. Let $i=x_0$ so $x\in I_i$. Then $x_1x_2\cdots \in X_i$ so $f(x)\in A_i \subset B_i$ and therefore $\phi(x)=0$.

To prove property (ii) we choose a point $x\not\in A$ such that $f(x)\in A$. There are two cases. First, suppose $x\not\in C$. If $x$ is in an increasing lap of $f$ then $f(x)\not\in C$ either, contradicting $f(x)\in A$. But if $x$ is not in an increasing lap of $f$ then $\phi(x)>0$ by the definition of $\phi$, so we are done. Second, suppose $x\in C\setminus A$. Then we can write $x = e(x_0 x_1 \cdots)$ with $x_0 x_1\cdots \in 5^\N\setminus X_i$. Let $i=x_0$, so $I_i$ is the lap containing $x$. Then $x_1 x_2 \cdots \not\in X_i $, so $f(x)\not\in A_i$. But we supposed $f(x)\in A$, so by~\eqref{BcapA} we have $f(x)\not\in B_i$. Therefore $\phi(x)>0$ by the definition of $\phi$.

To prove property (iii) pick any point $y\in I$. Since the sets $B_i$ cover $I$ we can find $i$ with $y\in B_i$. Then we take $x=f|_{I_i}^{-1}(y)\in I_i$. Then $\phi(y)=0$ by the definition of $\phi$.
\end{proof}

\begin{proof}[Proof of Theorem~\ref{th:sq-nonBorel}]
By Theorem~\ref{th:shift-nonBorel} the subshift $(X,\sigma)$ has a point $\omega_0$ such that $\sa(\omega_0)\subset X$ is not Borel. By Proposition~\ref{prop:embedding} there is a continuous surjection $F:I^2\to I^2$ and a closed subset $S\subset I^2$ such that $F^{-1}(S)=F(S)=S$ and $(S,F)$ is conjugate to $(X,\sigma)$. Let $E:X\to S \subset I^2$ denote the conjugacy and let $x_0 = E(\omega_0)$. The backward orbit branches of $x_0$ in the system $(I^2,F)$ are the same as the backward orbit branches in the subsystem $(S,F)$, and these correspond through the conjugacy to the backward orbit branches of $\omega_0$ in $(X,\sigma)$. Therefore $\sa(x_0)=E(\sa(\omega_0))$. Since $E$ is a topological embedding we conclude that $\sa(x_0)\subset I^2$ is not Borel.
\end{proof}

\begin{remark}\label{rem:embedding}
The only property of the subshift $(X,\sigma)$ which was really needed in the proof of Proposition~\ref{prop:embedding} was surjectivity $\sigma(X)=X$. Thus, the same proof technique leads to the following general embedding result.
\end{remark}

\begin{proposition}
Any subshift $X\subseteq\{0,\ldots,r-1\}^{\N}$, $r\geq2$, with $\sigma(X)=X$ can be embedded as a totally invariant subsystem of a surjective map on the square.
\end{proposition}
\begin{proof}
The proof is essentially the same as the proof of Proposition~\ref{prop:embedding}. To accomodate a shift space with $r$ symbols we divide the square into $2r-1$ vertical strips and use for the base map the full $2r-1$ horseshoe. No other significant changes are needed.
\end{proof}

\begin{remark}
In dimension one, the $\sa$-limit sets of an interval map $f:I\to I$ are always Borel and can only occur at or below the second level of the Borel hierarchy (they are simultaneously $F_\sigma$ and $G_\delta$, see~\cite{HantakRoth}). But the possibility to embed subshifts into the square without the restrictions imposed by the order structure of $\mathbb{R}$ gives much more flexibility to dynamical systems in dimension two than in dimension one. This gives some explanation of why their special $\alpha$-limit sets can be more complex.
\end{remark}

\end{document}